\documentclass[10pt]{article}

\usepackage[normalem]{ulem}
\usepackage[english]{babel}

\usepackage{url}
\usepackage[colorlinks = true]{hyperref}
\usepackage{fullpage,amsmath,amsthm,amssymb,amsbsy,mathtools}
\usepackage{paralist}
\usepackage{xcolor}
\usepackage{color}
\usepackage{graphicx}
\usepackage{comment}

\usepackage{fancyhdr}
\usepackage{cite}
\usepackage{cleveref}

\newtheorem{thm}{Theorem}
\newtheorem{prop}{Proposition}
\newtheorem{defi}{Definition}

\newtheorem{lem}{Lemma}
\theoremstyle{remark}
\newtheorem{remark}{Remark}

\newcommand{\R}{\mathbb{R}}

\newcommand{\vct}[1]{\boldsymbol{#1}}
\newcommand{\mtx}[1]{\boldsymbol{#1}}

\newcommand{\T}{\mathrm{T}}

\newcommand{\rank}{\operatorname{rank}}

\def \lg        {\langle}
\def \rg        {\rangle}
\def \vec       {\operatorname*{vec}}
\newcommand{\set}[1]{\mathcal{#1}}

\DeclareMathOperator*{\minimize}{\text{minimize}}

\DeclareMathOperator*{\argmin}{\text{arg~min}}

\def \calC {\set{C}}

\newcommand{\vq}{\vct{q}}

\newcommand{\vx}{\vct{x}}
\newcommand{\vy}{\vct{y}}
\newcommand{\vz}{\vct{z}}

\newcommand{\vdelta}{\vct{\delta}}
\newcommand{\vzero}{\vct{0}}

\newcommand{\mD}{\mtx{D}}

\newcommand{\mG}{\mtx{G}}

\newcommand{\mP}{\mtx{P}}
\newcommand{\mQ}{\mtx{Q}}

\newcommand{\mU}{\mtx{U}}
\newcommand{\mV}{\mtx{V}}
\newcommand{\mW}{\mtx{W}}
\newcommand{\mX}{\mtx{X}}
\newcommand{\mY}{\mtx{Y}}

\newcommand{\mDelta}{\mtx{\Delta}}

\newcommand{\mSigma}{\mtx{\Sigma}}

\newcommand{\mId}{{\bf I}}

\newcommand{\mzero}{{\bf 0}}

\newcommand{\setC}{\set{C}}

\graphicspath{{./figs/}}

\newlength{\imgwidth}
\def \lg        {\langle}
\def \rg        {\rangle}

\usepackage[framemethod=tikz]{mdframed}

\usepackage{pifont}  

\usepackage{wrapfig,floatflt}
\usepackage[caption=false,font=normalsize,labelfont=sf,textfont=sf]{subfig}
\usepackage{enumitem,cleveref}
\usepackage{algorithm}
\usepackage{algorithmic}        %
\title{Global Optimality in Distributed Low-rank Matrix Factorization}
\author{Zhihui Zhu, Qiuwei Li, Xinshuo Yang, Gongguo Tang, and Michael B. Wakin\footnote{The first two authors contributed equally to this work.} \\[3mm]  \normalsize Department of Electrical Engineering, Colorado School of Mines}
\begin{document}

\maketitle

\begin{abstract}
We study the convergence of a variant of distributed gradient descent (DGD) on a distributed low-rank matrix approximation problem wherein some optimization variables are used for consensus (as in classical DGD) and some optimization variables appear only locally at a single node in the network. We term the resulting algorithm DGD+LOCAL. Using algorithmic connections to gradient descent and geometric connections to the well-behaved landscape of the centralized low-rank matrix approximation problem, we identify sufficient conditions where DGD+LOCAL is guaranteed to converge with exact consensus to a global minimizer of the original centralized problem. For the distributed low-rank matrix approximation problem, these guarantees are stronger---in terms of consensus and optimality---than what appear in the literature for classical DGD and more general problems.
\end{abstract}

\section{Introduction}

A promising line of recent literature has examined the nonconvex objective functions that arise when certain matrix optimization problems are solved in factored form, that is, when a low-rank optimization variable $\mX$ is replaced by a product of two thin matrices $\mU \mV^\T$ and the optimization proceeds jointly over $\mU$ and $\mV$~\cite{zhu2018global,zhu2017global,Chi2018,li2018non,ChenChi2018Harnessing,GeEtAl2017No,SunLuo2016Guaranteed}. In many cases, a study of the geometric landscape of these objective functions reveals that---despite their nonconvexity---they possess a certain favorable geometry. In particular, many of the resulting objective functions ($i$) satisfy the {\em strict saddle property}~\cite{GeEtAl2015Escaping,SunEtAl2015When}, where every critical point is either a local minimum or is a strict saddle point, at which the Hessian matrix has at least one negative eigenvalue, and ($ii$) have no spurious local minima (every local minimum corresponds to a global minimum).

One such problem---which is both of fundamental importance and representative of structures that arise in many other machine learning problems~\cite{koren2009matrix}---is the low-rank matrix approximation problem, where given a data matrix $\mY$ the objective is to minimize $\| \mU \mV^\T - \mY \|_F^2$. As we explain in Theorem~\ref{thm:centerss}, building on recent analysis in~\cite{nouiehed2018learning} and~\cite{zhu2017global}, this problem satisfies the strict saddle property and has no spurious local minima.

In parallel with the recent focus on the favorable geometry of certain nonconvex landscapes, it has been shown that a number of local search algorithms have the capability to avoid strict saddle points and converge to a local minimizer for problems that satisfy the strict saddle property~\cite{lee2016gradient,lee2017first,jin2017escape,royer2018newton}. As stated in~\cite{lee2017first} and as we summarize in Theorems~\ref{thm:avoid saddle gd} and~\ref{thm:boundedgd}, gradient descent when started from a random initialization is one such algorithm. For problems such as low-rank matrix approximation that have no spurious local minima, converging to a local minimizer means converging to a global minimizer.

To date, the geometric and algorithmic research described above has largely focused on {\em centralized optimization}, where all computations happen at one ``central'' node that has full access, for example, to the data matrix $\mY$.

In this work, we study the impact of {\em distributing} the factored optimization problem, such as would be necessary if the data matrix $\mY$ in low-rank matrix approximation were partitioned into submatrices $\mY = \begin{bmatrix} \mY_1 & \mY_2 & \cdots & \mY_J \end{bmatrix}$, each of which was available at only one node in a network. By similarly partitioning the matrix $\mV$, one can partition the objective function
\begin{equation}
\| \mU \mV^\T - \mY \|_F^2 = \sum_{j=1}^J \| \mU\mV_j^\T - \mY_j \|_F^2.
\label{eq:pca1}
\end{equation}
As we discuss, one can attempt to minimize the resulting objective, in which the matrix $\mU$ appears in every term of the summation, using techniques similar to classical distributed algorithms such as distributed gradient descent (DGD)~\cite{nedic2009distributed}. These algorithms, however, involve creating local copies $\mU^1, \mU^2, \dots, \mU^J$ of the optimization variable $\mU$ and iteratively sharing updates of these variables with the aim of converging to a consensus where (exactly or approximately) $\mU^1 = \mU^2 = \cdots = \mU^J$.

In this paper we study a straightforward extension of DGD for solving such problems. This extension, which we term DGD+LOCAL, resembles classical DGD in that each node $j$ has a local {\em copy} $\mU^j$ of the optimization variable $\mU$ as described above. Additionally, however, each node has a local {\em block} $\mV_j$ of the partitioned optimization variable $\mV$, and this block exists only locally at node $j$ without any consensus or sharing among other nodes.

We present a geometric framework for analyzing the convergence of DGD+LOCAL in such problems. Our framework relies on a straightforward conversion which reveals (for example in the low-rank matrix approximation problem) that DGD+LOCAL as described above is equivalent to running conventional gradient descent on the objective function
\begin{equation}
\sum_{j=1}^J \left(\| \mU^j \mV_j^\T - \mY_j \|_F^2 + \sum_{i=1}^J w_{ji} \|\mU^j - \mU^i\|_F^2\right),
\label{eq:pca2}
\end{equation}
where $w_{ji}$ are weights inherited from the DGD+LOCAL iterations. This objective function~\eqref{eq:pca2} differs from the original objective function~\eqref{eq:pca1} in two respects: it contains more optimization variables, and it includes a quadratic regularizer to encourage consensus. Although the geometry of~\eqref{eq:pca1} is understood to be well-behaved, new questions arise about the geometry of~\eqref{eq:pca2}: Does it contain new critical points (local minima that are not global, saddle points that are not strict)? And on the consensus subspace, where $\mU^1 = \mU^2 = \cdots = \mU^J$, how do the critical points of~\eqref{eq:pca2} relate to the critical points of~\eqref{eq:pca1}? We answer these questions and build on the algorithmic results for gradient descent to identify in Theorem~\ref{thm:mainRevisedfj} sufficient conditions where DGD+LOCAL is guaranteed to converge to a point that ($i$) is exactly on the consensus subspace, and ($ii$) coincides with a global minimizer of problem~\eqref{eq:pca1}. Under these conditions, the distributed low-rank matrix approximation problem is shown to enjoy the same geometric and algorithmic guarantees as its well-behaved centralized counterpart.

For the distributed low-rank matrix approximation problem, these guarantees are stronger than what appear in the literature for classical DGD and more general problems. In particular, we show exact convergence to the consensus subspace with a fixed DGD+LOCAL stepsize, which in more general works is accomplished only with diminishing DGD stepsizes for convex~\cite{chen2012fast,jakovetic2014fast} and nonconvex~\cite{zeng2018nonconvex} problems or by otherwise modifying DGD as in the EXTRA algorithm~\cite{shi2015extra}. Moreover, we show convergence to a global minimizer of the original centralized nonconvex problem. Until recently, existing DGD results either considered convex problems~\cite{chen2012fast,jakovetic2014fast} or showed convergence to stationary points of nonconvex problems~\cite{zeng2018nonconvex}. Very recently, it was also shown~\cite{daneshmand2018second} that with an appropriately small stepsize, DGD can converge to an arbitrarily small neighborhood of a second-order critical point for general nonconvex problems with additional technical assumptions. Our work differs from~\cite{daneshmand2018second} in our use of DGD+LOCAL (rather than DGD) and our focus on one specific problem where we can establish stronger guarantees of exact global optimality and exact consensus without requiring an arbitrarily small (or diminishing) stepsize.

Our main results on distributed low-rank matrix factorization are presented in Section~\ref{sec:theorypca}. These results build on several more general algorithmic and geometric results that we first establish in Section~\ref{sec:theoryunc}. The results from Section~\ref{sec:theoryunc} may have broader applicability, and the geometric and algorithmic discussions in Section~\ref{sec:theoryunc} may have independent interest from one another.

\section{General Analysis of DGD+LOCAL}
\label{sec:theoryunc}

Consider a centralized minimization problem that can be written in the form
\begin{align}
\minimize_{\vx, \vy} f(\vx,\vy) = \sum_{j=1}^J f_j(\vx,\vy_j),
\label{eq:centeralized problem}\end{align}
where $\vy = \begin{bmatrix}\vy_1^\T & \cdots & \vy_J^\T\end{bmatrix}^\T$. Here $\vx$ is the common variable in all of the objective functions $\{f_j\}_{j\in[J]}$ and $\vy_j$ is the variable only corresponding to $f_j$.

The standard DGD algorithm~\cite{nedic2009distributed} is stated for problems of the form
\begin{align*}
\minimize_{\vx} f(\vx) = \sum_{j=1}^J f_j(\vx),
\end{align*}
and for such problems it involves updates of the form
\begin{align*}
\vx^j(k+1) &= \sum_{i=1}^J \left(\widetilde w_{ji}\vx^i(k)\right) - \mu  \nabla_{\vx} f_j(\vx^j (k)),
\end{align*}
where $\{\widetilde w_{ji}\}$ are a set of  symmetric  nonnegative weights, and $\widetilde w_{ji}$ is positive if and only if nodes $i$ and $j$ are neighbors in the network or $i=j$. Throughout this paper, we will make the common assumption~\cite{mokhtari2017network} that
\begin{align}
\sum_{i =1}^J \widetilde w_{ji} = 1 ~\text{for all}~ j\in[J].
\label{eq:sumto1}
\end{align}

A very natural extension of DGD to problems of the form~\eqref{eq:centeralized problem}---which involve local {\em copies} of the shared variable $\vx$ and local {\em partitions} of the variable $\vy$---is to perform the updates
\begin{align}
\vx^j(k+1) &= \sum_{i=1}^J \left(\widetilde w_{ji}\vx^i(k)\right) - \mu  \nabla_{\vx} f_j(\vx^j (k),\vy_j(k)), \nonumber \\
\vy_j(k+1) & = \vy_j(k) - \mu \nabla_{\vy}f_j(\vx^j(k),\vy_j(k)).
\label{eq:DGDtemplate}
\end{align}
Because we are interested in solving problems of the form~\eqref{eq:centeralized problem}, we refer to~\eqref{eq:DGDtemplate} as DGD+LOCAL throughout this paper. We note that DGD+LOCAL is not equivalent to algorithm would obtain by applying classical DGD to reach consensus over the concatenated variables $\vx$ and $\vy$ as this would require each node to maintain a local copy of the entire vector $\vy$. For the same reason, DGD+LOCAL is not equivalent to the blocked variable problem described in~\cite{notarnicola2017distributed}.

\subsection{Relation to Gradient Descent}
\label{sec:dgdgd}

Note that we can rewrite the first equation in~\eqref{eq:DGDtemplate} as
\begin{align*}
\vx^j(k+1) &= (\sum_{i =1}^J \widetilde w_{ji})\vx^j(k) - \mu \left( \nabla_{\vx} f_j(\vx^j (k),\vy_j(k))  + \sum_{i\neq j} \frac{\widetilde w_{ji}}{\mu}(\vx^j(k) - \vx^i(k))   \right) \\
&= \vx^j(k) - \mu \left( \nabla_{\vx} f_j(\vx^j (k),\vy_j(k))  + \sum_{i\neq j} \frac{\widetilde w_{ji}}{\mu}(\vx^j(k) - \vx^i(k))   \right).
\end{align*}
In the second line, we have used the assumption~\eqref{eq:sumto1}. Thus,  by defining $\{w_{ji}\}$ such that
\begin{equation}
w_{ji}=w_{ij} = \begin{cases} \frac{\widetilde w_{ji}}{4\mu}, & i \neq j, \\ 0, & i=j, \end{cases}
\label{eq:wtildetow}
\end{equation}
we see that DGD+LOCAL~\eqref{eq:DGDtemplate} is equivalent to applying standard gradient descent (with stepsize $\mu$) to the problem
\begin{equation}\begin{split}
&\minimize_{\vz} g(\vz) = \sum_{j=1}^J \left(f_j(\vx^j,\vy_j) + \sum_{i=1}^J w_{ji} \|\vx^j - \vx^i\|_2^2\right),
\end{split}\label{eq:DGD problem}\end{equation}
where $\vz = (\vx^1,\ldots,\vx^J,\vy_1,\ldots,\vy_J)$ and $\mW = \{w_{ji}\}$ is a $J \times J$ connectivity matrix with nonnegative entries defined in~\eqref{eq:wtildetow} and zeros on the diagonal.

\subsection{Algorithmic Analysis}
\label{sec:dgdalg}

We are interested in understanding the convergence of the gradient descent algorithm when it is applied to minimizing $g(\vz)$ in~\eqref{eq:DGD problem}; as we have argued in Section~\ref{sec:dgdgd}, this is equivalent to running the DGD+LOCAL algorithm~\eqref{eq:DGDtemplate} to minimize the objective function $f(\vx,\vy)$ in~\eqref{eq:centeralized problem}.

Under certain conditions, we can guarantee that gradient descent will converge to a second-order critical point of the objective function $g(\vz)$ in~\eqref{eq:DGD problem}. The proof relies on certain properties of the functions $f_j$ comprising~\eqref{eq:centeralized problem}. We first describe these properties before providing the convergence result.

\subsubsection{Objective Function Properties and Convergence of Gradient Descent}

The first property concerns the assumption that each $f_j$ comprising~\eqref{eq:centeralized problem} has Lipschitz gradient. In this case we can also argue that $g$ in~\eqref{eq:DGD problem} has Lipschitz gradient.

\begin{prop}
Let $f(\vx,\vy) = \sum_{j=1}^J f_j(\vx,\vy_j)$ be an objective function as in~\eqref{eq:centeralized problem} and let $g(\vz)$ be as in \eqref{eq:DGD problem} with $\vz = (\vx^1,\ldots,\vx^J,\vy_1,\ldots,\vy_J)$. Suppose that each $f_j$ has Lipschitz gradient, i.e., $\nabla f_j$ is Lipschitz continuous with constant $L_j>0$. Then $\nabla g$ is Lipschitz continuous with constant
\[
L_{g} = L + \frac{2\omega}{\mu},
\]
where $L := \max_j L_j$,  $ \omega :=\sum_{i\neq j }^J  \widetilde w_{ji}$, and $\widetilde w_{ji}$ and $\mu$ are the DGD+LOCAL weights and stepsize as in~\eqref{eq:DGDtemplate}.
\label{prop:lip}
\end{prop}

Proposition~\ref{prop:lip} is proved in Appendix~\ref{sec:prooflip}.

The second property concerns the following {\L}ojasiewicz inequality, which arises in the convergence analysis of gradient descent.
\begin{defi}\cite{attouch2009convergence}\label{def:KL}
Assume that $h:\R^n \rightarrow \R$ is continuously differentiable. Then $h$ is said to satisfy the {\L}ojasiewicz inequality, if for any critical point $\overline{\vx}$ of $h(\vx)$, there exist $\delta>0,~\theta\in[0,1),~C_1>0$ such that
	\[
	\left|h(\vx) - h(\overline{\vx})\right|^{\theta} \leq C_1 \|\nabla h(\vx)\|,~~\forall~\vx\in B(\overline{\vx}, \delta).
	\]
	Here $\theta$ is often referred to as the KL exponent.
\end{defi}
This {\L}ojasiewicz inequality (or a more general Kurdyka-{\L}ojasiewicz (KL) inequality for the general nonsmooth problems) characterizes the local geometric properties of the objective function around its critical points and has proved useful for convergence analysis \cite{attouch2009convergence,bolte2014proximal}. The {\L}ojasiewicz inequality (or KL inequality) is very general and holds for most problems in engineering. For example, every analytic function satisfies this {\L}ojasiewicz inequality, but each function may have different {\L}ojasiewicz exponent $\theta$ which determines the convergence rate; see \cite{attouch2009convergence,bolte2014proximal} for the details on this.

A general result for convergence of gradient descent to first-order critical point for a function satisfying the {\L}ojasiewicz inequality is as follows.\footnote{The result in~\cite{attouch2009convergence} is stated for the proximal method, but the result can be extended to gradient descent as long as  $\mu<\frac{1}{L}$.}
\begin{thm}\cite{attouch2009convergence}
Suppose $\inf_{\R^n} h>-\infty$ and $h$ satisfies the {\L}ojasiewicz inequality. Also assume $\nabla h$ is Lipschitz continuous with constant $L>0$. Let $\{\vx(k)\}$ be the sequence generated by gradient descent $\vx({k+1}) =  \vx(k) - \mu \nabla h(\vx(k))$ with $\mu<\frac{1}{L}$. Then if the sequence $\{\vx(k)\}$ is bounded, it converges to a critical point of $h$.
\label{thm:convergence gd with KL}
\end{thm}

The following result further characterizes the convergence behavior of gradient descent to a second-order critical point.
 \begin{thm}\cite{lee2016gradient}
		Suppose $h$ is a twice-continuously differentiable function and $\nabla h$ is Lipschitz continuous with constant $L>0$. Let $\{\vx(k)\}$ be the sequence generated by gradient descent $\vx({k+1}) =  \vx(k) - \mu \nabla h(\vx(k))$ with $\mu<\frac{1}{L}$. Suppose $\vx(0)$ is chosen randomly from a probability distribution supported on a set $S$ having positive measure. Then the sequence $\{\vx(k)\}$  almost surely avoids strict saddles, where the Hessian has at least one negative eigenvalue.
		\label{thm:avoid saddle gd}
	\end{thm}

Theorems~\ref{thm:convergence gd with KL} and~\ref{thm:avoid saddle gd} apply for functions $h$ that globally satisfy the {\L}ojasiewicz and Lipschitz gradient conditions. In some problems, however, one or both of these properties may be satisfied only locally. Nevertheless, under an assumption of bounded iterations---as is already made in Theorem~\ref{thm:convergence gd with KL}---it is possible to extend the first- and second-order convergence results to such functions. For example, one can extend Theorem~\ref{thm:convergence gd with KL} as follows by noting that the original derivation in~\cite{attouch2009convergence} used the {\L}ojasiewicz property only locally around limit points of the sequence $\{\vx(k)\}$.

\begin{thm}\cite{attouch2009convergence}
Suppose $\inf_{\R^n} h>-\infty$. For $\rho > 0$, let $B_\rho$ denote the open ball of radius $\rho$:
\[
B_\rho := \{ \vx: ~ \|\vx\|_2 < \rho \},
\]
and suppose $h$ satisfies the {\L}ojasiewicz inequality at all points $\vx \in B_\rho$. Also assume $\nabla h$ is Lipschitz continuous with constant $L>0$. Let $\{\vx(k)\}$ be the sequence generated by gradient descent $\vx({k+1}) =  \vx(k) - \mu \nabla h(\vx(k))$ with $\mu<\frac{1}{L}$. Suppose $\{\vx(k)\} \subseteq B_\rho$ and all limit points of $\{\vx(k)\}$ are in $B_\rho$. Then the sequence $\{\vx(k)\}$ converges to a critical point of $h$.
\label{thm:convergence gd with KL in ball}
\end{thm}

The following result establishes second-order convergence for a function with a locally Lipschitz gradient.

\begin{thm}
Let $\rho > 0$, and consider an objective function $h$ where:
\begin{enumerate}
\item $\inf_{\R^n} h>-\infty$,
\item $h$ satisfies the {\L}ojasiewicz inequality within $B_\rho$,
\item $h$ is twice-continuously differentiable, and
\item $\left| h\left(\vx\right)\right| \leq L_0$, $\left\Vert \nabla h\left(\vx\right)\right\Vert \leq L_1$, and $\left\Vert \nabla^{2}h(\vx) \right\Vert_{2}\leq L_2$ for all $\vx \in B_{2\rho}$.
\end{enumerate}
Suppose the gradient descent stepsize
\begin{equation}
\mu < \frac{1}{L_{2}+\frac{4L_{1}}{\rho}+\frac{\left(2+2\pi\right)L_{0}}{\rho^{2}}}.
\label{eq:steptilde}
\end{equation}
Suppose $\vx(0)$ is chosen randomly from a probability distribution supported on a set $S \subseteq B_\rho$ with $S$ having positive measure, and suppose that under such random initialization, there is a positive probability that the sequence $\{\vx(k)\}$ remains bounded in $B_\rho$ and all limit points of $\{\vx(k)\}$ are in $B_\rho$.

Then conditioned on observing that $\{\vx(k)\} \subseteq B_\rho$ and all limit points of $\{\vx(k)\}$ are in $B_\rho$, gradient descent converges to a critical point of $h$, and the probability that this critical point is a strict saddle point is zero.
\label{thm:boundedgd}
\end{thm}

Theorem~\ref{thm:boundedgd} is proved in Appendix~\ref{sec:proofboundedgd}.

\subsubsection{Convergence Analysis of DGD+LOCAL}

As described in the following theorem, under certain conditions, we can guarantee that the DGD+LOCAL algorithm~\eqref{eq:DGDtemplate} (which is equivalent to gradient descent applied to minimizing $g(\vz)$ in~\eqref{eq:DGD problem}) will converge to a second-order critical point of the objective function $g(\vz)$.
\begin{thm}
Let $f(\vx,\vy) = \sum_{j=1}^J f_j(\vx,\vy_j)$ be an objective function as in~\eqref{eq:centeralized problem} and let $g(\vz)$ be as in \eqref{eq:DGD problem} with $\vz = (\vx^1,\ldots,\vx^J,\vy_1,\ldots,\vy_J)$. Suppose each $f_j$ satisfies $\inf_{\R^n} f_j>-\infty$, is twice continuously-differentiable, and has Lipschitz gradient, i.e., $\nabla f_j$ is Lipschitz continuous with constant $L_j>0$. Suppose $g$ satisfies the {\L}ojasiewicz inequality. Let $L := \max_j L_j$, and let $\widetilde w_{ji}$ and $\mu$ be the DGD+LOCAL weights and stepsize as in~\eqref{eq:DGDtemplate}.

Assume $\omega := \max_{j}\sum_{i\neq j} \widetilde w_{ji}<\frac{1}{2}$. Let $\{ \vz(k) \}$ be the sequence generated by the DGD+LOCAL algorithm in \eqref{eq:DGDtemplate} with
\begin{align}
	\mu < \frac{1 - 2\omega}{L}
\label{eq:stepsize requirement}
\end{align}
and with random initialization from a probability distribution supported on a set $S$ having positive measure. Then if the sequence $\{\vz(k)\}$ is bounded, it almost surely converges to a second-order critical point of the objective function in \eqref{eq:DGD problem}.
\label{thm:dgdconvergegeneric}
\end{thm}

\begin{proof}
Recall that running the DGD+LOCAL algorithm~\eqref{eq:DGDtemplate} to minimize the objective function $f(\vx,\vy)$ in~\eqref{eq:centeralized problem} is equivalent to running gradient descent on $g(\vz)$ in~\eqref{eq:DGD problem}. The proof is completed by invoking \Cref{thm:convergence gd with KL} and \Cref{thm:avoid saddle gd} with $h$ replaced by $g$. From Proposition~\ref{prop:lip}, we have that $\nabla g$ is Lipschitz continuous with constant $L_g = L + \frac{2\omega}{\mu}$, and so choosing $\mu$ to satisfy~\eqref{eq:stepsize requirement} ensures that $\mu < \frac{1}{L_g}$ as required in \Cref{thm:convergence gd with KL} and \Cref{thm:avoid saddle gd}.
\end{proof}

\begin{remark}
The requirement that the DGD+LOCAL stepsize $\mu = O(\frac{1}{L})$ also appears in the convergence analysis of DGD in \cite{yuan2016convergence,zeng2018nonconvex}.
\end{remark}

\begin{remark}
The function $g$ is guaranteed to satisfy the {\L}ojasiewicz inequality, for example, if every $f_j$ is semi-algebraic, because this will imply that $g$ is semi-algebraic, and every semi-algebraic function satisfies the {\L}ojasiewicz inequality. 
\end{remark}

\begin{remark}
In order to satisfy~\eqref{eq:stepsize requirement}, it must hold that $\omega < \frac{1}{2}$. In the case where the DGD+LOCAL weight matrix $\widetilde\mW$ is symmetric and doubly stochastic (i.e., $\widetilde\mW$ has nonnegative entries and each of its rows and columns sums to $1$), this condition is equivalent to requiring that each diagonal element of $\widetilde\mW$ is larger than $\frac{1}{2}$. Given any symmetric and doubly stochastic matrix $\widetilde\mW$, one can design a new weight matrix $(\widetilde\mW + \mId)/2$ that satisfies this requirement. This strategy is also mentioned at the end of \cite[Section 2.1]{yuan2016convergence}.
\end{remark}

We also have the following DGD+LOCAL convergence result when the functions $f_j$ have only a locally Lipschitz gradient.

\begin{thm}
Let $f(\vx,\vy) = \sum_{j=1}^J f_j(\vx,\vy_j)$ be an objective function as in~\eqref{eq:centeralized problem} and let $g(\vz)$ be as in \eqref{eq:DGD problem} with $\vz = (\vx^1,\ldots,\vx^J,\vy_1,\ldots,\vy_J)$. Let $\rho > 0$ and suppose each $f_j$ satisfies
\begin{enumerate}
\item $\inf_{\R^n} f_j>-\infty$,
\item $f_j$ is twice-continuously differentiable, and
\item $\left| f_j\left(\vx,\vy_j\right)\right| \leq L_{0,j}$, $\left\Vert \nabla f_j\left(\vx,\vy_j\right)\right\Vert \leq L_{1,j}$, and $\left\Vert \nabla^{2}f_j(\vx,\vy) \right\Vert_{2}\leq L_{2,j}$ for all $(\vx,\vy_j) \in B_{2\rho}$.
\end{enumerate}
Suppose also that $g$ satisfies the {\L}ojasiewicz inequality within $B_\rho$. Let $\widetilde w_{ji}$ and $\mu$ be the DGD+LOCAL weights and stepsize as in~\eqref{eq:DGDtemplate}. Assume $\omega := \max_{j}\sum_{i\neq j} \widetilde w_{ji}< \frac{1}{2}$. Let $\{ \vz(k) \}$ be the sequence generated by the DGD+LOCAL algorithm in \eqref{eq:DGDtemplate} with
\begin{equation}
\mu < \frac{1-2\omega}{\max_j L_{2,j}+\frac{4L_{1,j}}{\rho}+\frac{\left(2+2\pi\right)L_{0,j}}{\rho^{2}}}.
\label{eq:stepsize requirementballfj}
\end{equation}
Suppose $\vz(0)$ is chosen randomly from a probability distribution supported on a set $S \subseteq B_\rho$ with $S$ having positive measure, and suppose that under such random initialization, there is a positive probability that the sequence $\{\vz(k)\}$ remains bounded in $B_\rho$ and all limit points of $\{\vz(k)\}$ are in $B_\rho$.

Then conditioned on observing that $\{\vz(k)\} \subseteq B_\rho$ and all limit points of $\{\vz(k)\}$ are in $B_\rho$, DGD+LOCAL converges to a critical point of the objective function in \eqref{eq:DGD problem}, and the probability that this critical point is a strict saddle point is zero.
\label{thm:dgdconvergegenericballfj}
\end{thm}

\begin{proof}
Recall that running the DGD+LOCAL algorithm~\eqref{eq:DGDtemplate} to minimize the objective function $f(\vx,\vy)$ in~\eqref{eq:centeralized problem} is equivalent to running gradient descent on $g(\vz)$ in~\eqref{eq:DGD problem}. Similar to the approach taken in proving Theorem~\ref{thm:boundedgd}, to deal with the local Lipschitz condition, the proof involves constructing a function $\widetilde{g}$ such that $\widetilde{g}(\vz) = g(\vz)$ for all $\vz \in B_{\rho}$ but where $\widetilde{g}$ has a globally Lipschitz gradient.

To do this, recall the window function $w$ defined in Appendix~\ref{sec:proofboundedgd}. Now, recall that
\[
g(\vz) = \sum_{j=1}^J \left(f_j(\vx^j,\vy_j) + \sum_{i=1}^J w_{ji} \|\vx^j - \vx^i\|_2^2\right)
\]
and define
\begin{equation}
\widetilde{g}\left(\vz\right) = \sum_{j=1}^J \left(\widetilde{f}_j(\vx^j,\vy_j) + \sum_{i=1}^J w_{ji} \|\vx^j - \vx^i\|_2^2\right),
\label{eq:gtilde}
\end{equation}
where
\[
\widetilde{f}_j(\vx^j,\vy_j) = f_j(\vx^j,\vy_j) w(\begin{bmatrix} (\vx^j)^\T & \vy_j^\T \end{bmatrix}^\T).
\]
Since $\widetilde{f}_j(\vx^j,\vy_j) = f_j(\vx^j,\vy_j)$ for $(\vx^j,\vy_j) \in B_\rho$, we have that $\widetilde{g}\left(\vz\right) = g(\vz)$ for all $\vz \in B_\rho$.

We have the following properties for $\widetilde{g}$:
\begin{itemize}
\item Since $g = \widetilde{g}$ in $B_\rho$, $\widetilde{g}$ satisfies the {\L}ojasiewicz inequality in $B_\rho$.
\item Since $f_j \in C^2$ for all $j$ and $w \in C^2$, $\widetilde{g} \in C^2$.
\item Since $\inf_{\R^n} f_j>-\infty$ for all $j$ and $\inf_{\R^n} w>-\infty$, $\inf_{\R^n} \widetilde{g}>-\infty$.
\item To globally bound the Lipschitz constant of the gradient of $\widetilde{g}$, note that
\begin{eqnarray*}
\left\Vert \nabla^{2}\widetilde{f}_j\right\Vert  & = & \left\Vert w\cdot\nabla^{2}f_j+\nabla f_j \cdot \left(\nabla w\right)^{\T}+\nabla w \cdot \left(\nabla f_j\right)^{\T}+f_j\cdot\nabla^{2}w\right\Vert \\
 & \leq & \left| w\right| \left\Vert \nabla^{2}f_j\right\Vert +2\left\Vert \nabla w\right\Vert \left\Vert \nabla f_j\right\Vert +\left| f_j\right| \left\Vert \nabla^{2}w\right\Vert \\
 & \leq & L_{2,j}+\frac{4L_{1,j}}{\rho}+\frac{\left(2+2\pi\right)L_{0,j}}{\rho^{2}} \quad \text{for all}~(\vx^j,\vy_j).
\end{eqnarray*}
Therefore, given the form of $\widetilde{g}$ in~\eqref{eq:gtilde}, we can conclude from Proposition~\ref{prop:lip} that globally, $\nabla \widetilde{g}$ is Lipschitz continuous with constant
\[
L_{\widetilde{g}} = \left( \max_j L_{2,j}+\frac{4L_{1,j}}{\rho}+\frac{\left(2+2\pi\right)L_{0,j}}{\rho^{2}} \right) + \frac{2\omega}{\mu}.
\]
\end{itemize}

Now consider the gradient descent algorithm with stepsize $\mu$ satisfying~\eqref{eq:stepsize requirementballfj}. Define
\begin{align*}
T_g = \{\vz(0) \in B_\rho:&~\text{all}~\{\vz(k)\} \subseteq B_\rho~\text{and all limit points of}~\{\vz(k)\}~\text{are in}~B_\rho\\
&~\text{when gradient descent is run on}~g~\text{starting at}~\vz(0)\}
\end{align*}
and
\begin{align*}
T_{\widetilde{g}} = \{\vz(0) \in B_\rho:&~\text{all}~\{\vz(k)\} \subseteq B_\rho~\text{and all limit points of}~\{\vz(k)\}~\text{are in}~B_\rho\\
&~\text{when gradient descent is run on}~\widetilde{g}~\text{starting at}~\vz(0)\}.
\end{align*}

Similarly, define
\[
\Sigma_g = \{\vz(0) \in B_\rho:~\{\vz(k)\}~\text{converges to a strict saddle when gradient descent is run on}~g~\text{starting at}~\vz(0)\}
\]
and
\[
\Sigma_{\widetilde{g}} = \{\vz(0) \in B_\rho:~\{\vz(k)\}~\text{converges to a strict saddle when gradient descent is run on}~\widetilde{g}~\text{starting at}~\vz(0)\}.
\]
Using the above properties, we see that Theorem~\ref{thm:avoid saddle gd} can be applied to $\widetilde{g}$, and so we conclude that $\Sigma_{\widetilde{g}}$ has measure zero.

Now, after running gradient descent on $g$ from a random initialization as in the theorem statement, condition on observing that $\{\vz(k)\} \subseteq B_\rho$ and all limit points of $\{\vz(k)\}$ are in $B_\rho$, i.e., that $\vz(0) \in T_g$. Because $\{\vz(k)\} \subseteq B_\rho$ and all limit points of $\{\vz(k)\}$ are in $B_\rho$, and because $\{\vz(k)\}$ matches the sequence that would be obtained by running gradient descent on $\widetilde{g}$, we can apply Theorem~\ref{thm:convergence gd with KL in ball} to conclude that $\{\vz(k)\}$ converges to a critical point of $\widetilde{g}$, and since this critical point belongs to $B_\rho$ and $\widetilde{g} = g$ inside $B_\rho$, we conclude that this is also a critical point of $g$.

Finally, using the definition of conditional probability, we have
\begin{align*}
P(\vz(0) \in \Sigma_g | \vz(0) \in T_g)
&= \frac{P( \vz(0) \in \Sigma_g \cap T_g )}{P(\vz(0) \in T_g)} \\
&= \frac{P( \vz(0) \in \Sigma_{\widetilde{g}} \cap T_{\widetilde{g}} )}{P(\vz(0) \in T_g)},
\end{align*}
where the second equality follows from the fact that $\widetilde{g} = g$ inside $B_\rho$: if a sequence of iterations stays bounded inside $B_\rho$ and converges to a strict saddle when gradient descent is run on $g$, the same will hold when gradient descent is run on $\widetilde{g}$, and vice versa. Since $\Sigma_{\widetilde{g}}$ has zero measure and because $\vz(0)$ is chosen randomly from a probability distribution supported on a set $S \subseteq B_\rho$ with $S$ having positive measure, $P( \vz(0) \in \Sigma_{\widetilde{g}} \cap T_{\widetilde{g}} ) = 0$. Also, by assumption, $P(\vz(0) \in T_g) > 0$. Therefore, $P(\vz(0) \in \Sigma_g | \vz(0) \in T_g) = \frac{0}{\text{nonzero}} = 0$.
\end{proof}

\subsection{Geometric Analysis}

Section~\ref{sec:dgdalg} establishes that, under certain conditions, DGD+LOCAL will converge to a second-order critical point of the objective function $g(\vz)$ in~\eqref{eq:DGD problem}.

In this section, we are interested in studying the geometric landscape of the distributed objective function in~\eqref{eq:DGD problem} and comparing it to the geometric landscape of the original centralized objective function in~\eqref{eq:centeralized problem}. In particular, we would like to understand how the critical points of $g(\vz)$ in~\eqref{eq:DGD problem}) are related to the critical points of $f(\vx,\vy)$ in~\eqref{eq:centeralized problem}.

These problems differ in two important respects:
\begin{itemize}
\item The objective function in~\eqref{eq:DGD problem} involves more optimization variables than that in~\eqref{eq:centeralized problem}. Thus, the optimization takes place in a higher-dimensional space and there is the potential for new features to be introduced into the geometric landscape.
\item The objective function in~\eqref{eq:DGD problem} involves a quadratic regularization term that will promote consensus among the variables $\vx^1,\ldots,\vx^J$. This term is absent from~\eqref{eq:centeralized problem}. However, along the {\em consensus subspace} where $\vx^1 = \cdots = \vx^J$, this regularizer will be zero and the objective functions will coincide.
\end{itemize}
Despite these differences, we characterize below some ways in which the geometric landscapes of the two problems may be viewed as equivalent. These results may have independent interest from the specific DGD+LOCAL convergence analysis in Section~\ref{sec:dgdalg}.

Our first result establishes that if the sub-objective functions $f_j$ satisfy certain properties, the formulation \eqref{eq:DGD problem} does not introduce any new global minima outside of the consensus subspace.

\begin{prop}\label{prop:min f = min g for DGD} Let $f(\vx,\vy) = \sum_{j=1}^J f_j(\vx,\vy_j)$ be as in~\eqref{eq:centeralized problem}. Suppose the topology defined by $\mW$ is connected.
Also suppose there exist $\vx^\star$ (which is independent of $j$) and $\vy_j^\star, j\in[J]$ such that
	\begin{align}
	(\vx^\star,\vy_j^\star) \in\argmin_{\vx,\vy_j}f_j(\vx,\vy_j), \ \forall \ j\in [J].
	\label{eq:condition on f = g}\end{align}
	Then $g(\vz)$ defined in \eqref{eq:DGD problem} satisfies
	\[
	\min_{\vz} g(\vz) = \min_{\vx,\vy} f(\vx,\vy),
	\]
	and $g(\vz)$ achieves its global minimum only for $\vz$ with $\vx^1 = \cdots = \vx^J$.
\end{prop}

Proposition~\ref{prop:min f = min g for DGD} is proved in Appendix~\ref{sec:prooffgdgd}. We note that the assumption in Proposition~\ref{prop:min f = min g for DGD} is fairly strong, and while there are problems where it can hold, there are also many problems where it will not hold.

Proposition~\ref{prop:min f = min g for DGD} establishes that, in certain cases, there will exist no global minimizers of the distributed objective function $g(\vz)$ that fall outside of the consensus subspace. (Moreover, and also importantly, there will {\em exist} a global minimizer {\em on} the consensus subspace.) Also relevant is the question of whether there may exist any {\em other} types of critical points (such as local minima or saddle points) outside of the consensus subspace. Under certain conditions, the following proposition ensures that the answer is no.

\begin{prop}
	Let $f(\vx,\vy)$ be as in~\eqref{eq:centeralized problem} and $g(\vz)$ be as in \eqref{eq:DGD problem} with $\vz = (\vx^1,\ldots,\vx^J,\vy_1,\ldots,\vy_J)$. Suppose the matrix $\mW$ is connected and symmetric. Also suppose the gradient of $f_j$ satisfies the following symmetric property:	
	\begin{align}
	\langle \nabla_{\vx} f_j (\vx,\vy_j), \vx\rangle = \langle \nabla_{\vy_j} f_j (\vx,\vy_j), \vy_j\rangle
	\label{eq:symmetric property for gradientdgd}\end{align}
	for all $j\in[J]$.
	Then, any critical point of $g$ must satisfy $\vx^1 = \cdots = \vx^J $.
	\label{prop:nospuriouscritical for DGD}
\end{prop}

Proposition~\ref{prop:nospuriouscritical for DGD} is proved in Appendix~\ref{sec:proofnscdgd}.

Finally, we can also make a statement about the behavior of critical points that do fall on the consensus subspace.

\begin{thm}
	Let $\setC_f$ denote the set of critical points of \eqref{eq:centeralized problem}:
	\begin{align*}
	\setC_f: = \left\{\vx,\vy: \nabla f(\vx,\vy) = \vzero \right\},
\end{align*}
	and let $\setC_g$ denote the set of critical points of \eqref{eq:DGD problem}:
	\begin{align*}
	\setC_g: = \bigg\{\vz: \nabla g(\vz) = \vzero\bigg\}.
	\end{align*}
	Then, for any $\vz = (\vx^1,\ldots,\vx^J,\vy)\in\setC_g$ with $\vx^1 = \cdots = \vx^J = \vx$, we have $(\vx,\vy)\in \setC_f$. Furthermore, if $(\vx,\vy)$ is a strict saddle of $f$, then $\vz = (\vx,\dots,\vx,\vy)$ is also a strict saddle of $g$.
	\label{thm:strict saddle preserved for DGD}\end{thm}

The proof of Theorem~\ref{thm:strict saddle preserved for DGD} is in Appendix~\ref{sec:proofsspdgd}.

\section{Analysis of Distributed Matrix Factorization}
\label{sec:theorypca}

We now consider the prototypical low-rank matrix approximation in factored form, where given a data matrix $\mY \in\R^{n\times m}$, we seek to solve
\begin{align}
\minimize_{\mU \in \R^{n\times r},\mV \in \R^{m\times r}} \| \mU\mV^\T - \mY \|_F^2.
\label{eq:rankrBM}
\end{align}
Here $\mU \in \R^{n\times r}$ and $\mV \in \R^{m\times r}$ are tall matrices, and $r$ is chosen in advance to allow for a suitable approximation of $\mY$. In some of our results below, we will assume that the data matrix $\mY$ has rank at most~$r$.

One can solve problem~\eqref{eq:rankrBM} using local search algorithms such as gradient descent. Such algorithms do not require expensive SVDs, and the storage complexity for $\mU$ and $\mV$ scales with $(n+m)r$, which is smaller than $nm$ as for $\mY$. Unfortunately, problem~\eqref{eq:rankrBM} is nonconvex in the optimization variables $(\mU,\mV)$. Thus, the question arises of whether local search algorithms such as gradient descent actually converge to a global minimizer of~\eqref{eq:rankrBM}. Using geometric analysis of the critical points of problem~\eqref{eq:rankrBM}, however, it is possible to prove convergence to a global minimizer.

In Appendix~\ref{sec:geometry for pca}, building on analysis in~\cite{nouiehed2018learning}, we prove the following result about the favorable geometry of the nonconvex problem~\eqref{eq:rankrBM}.
\begin{thm}
For any data matrix $\mY$, every critical point (i.e., every point where the gradient is zero) of problem \eqref{eq:rankrBM} is either a global minimum or a strict saddle point, where the Hessian has at least one negative eigenvalue.
\label{thm:centerss}
\end{thm}
Such favorable geometry has been used in the literature to show that local search algorithms (particularly gradient descent with random initialization~\cite{lee2016gradient}) will converge to a global minimum of the objective function.

\subsection{Distributed Problem Formulation}

We are interested in generalizing the matrix approximation problem from centralized to distributed scenarios. To be specific, suppose the columns of the data matrix $\mY$ are distributed among $J$ nodes/sensors. Without loss of generality, partition the columns of $\mY$ as
\begin{align*}
\mY = \begin{bmatrix} \mY_1 & \mY_2 & \cdots & \mY_J \end{bmatrix},
\end{align*}
where for $j \in \{1,2,\dots,J\}$, matrix $\mY_j$ (which is stored at node $j$) has size $n \times m_j$, and where $m = \sum_{j=1}^J m_j$. Partitioning $\mV$ similarly as
\begin{align}
\mV = \begin{bmatrix}\mV_1^\T & \cdots & \mV_J^\T\end{bmatrix}^\T,
\label{eq:vpart}
\end{align}
where $\mV_j$ has size $m_j \times r$, we obtain the following optimization problem
\begin{align}
\minimize_{\mU,\mV_1,\dots,\mV_J} \sum_{j=1}^J \| \mU\mV_j^\T - \mY_j \|_F^2,
\label{eq:rankrBMconsensus}
\end{align}
which is exactly equivalent to~\eqref{eq:rankrBM}. Problem~\eqref{eq:rankrBMconsensus}, in turn, can be written in the form of problem~\eqref{eq:centeralized problem} by taking
\begin{align}
\vx = \text{vec}(\mU), ~~ \vy_j = \text{vec}(\mV_j), ~~ \text{and}~ f_j(\vx,\vy_j) = \| \mU\mV_j^\T - \mY_j \|_F^2.
\label{eq:xyfj}
\end{align}
Consequently, we can use the analysis from Section~\ref{sec:theoryunc} to study the performance of DGD+LOCAL~\eqref{eq:DGDtemplate} when applied to problem~\eqref{eq:rankrBMconsensus}.

For convenience, we note that in this context the DGD+LOCAL iterations~\eqref{eq:DGDtemplate} take the form
\begin{align}
\mU^j(k+1) &= \sum_{i=1}^J \left(\widetilde w_{ji}\mU^i(k)\right) - 2 \mu (\mU^j(k) \mV_j^\T(k) - \mY_j) \mV_j(k), \nonumber \\
\mV_j(k+1) & = \mV_j(k) - 2\mu (\mU^j(k) \mV_j^\T(k) - \mY_j)^\T \mU^j(k),
\label{eq:DGDtemplatePCA}
\end{align}
and the corresponding gradient descent objective function~\eqref{eq:DGD problem} takes the form
\begin{equation}
\begin{split}
&\minimize_{\vz} g(\vz) = \sum_{j=1}^J \left(\| \mU^j \mV_j^\T - \mY_j \|_F^2 + \sum_{i=1}^J w_{ji} \|\mU^j - \mU^i\|_F^2\right),
\end{split}
\label{eq:DGD problem PCA}
\end{equation}
where $\mU^1, \dots, \mU^J \in \R^{n\times r}$ are local copies of the optimization variable $\mU$; $\mV_1,\dots,\mV_J$ are a partition of $\mV$ as in~\eqref{eq:vpart}; and the weights $\{w_{ji}\}$ are determined by $\{\widetilde{w}_{ji}\}$ and $\mu$ as in~\eqref{eq:wtildetow}.

Problems~\eqref{eq:rankrBMconsensus} and~\eqref{eq:DGD problem PCA} (as special cases of problems~\eqref{eq:centeralized problem} and~\eqref{eq:DGD problem}, respectively) satisfy many of the assumptions required for the geometric and algorithmic analysis in Section~\ref{sec:theoryunc}. We use these facts in proving our main
result for the convergence of DGD+LOCAL on the matrix factorization problem.

\begin{thm}
Suppose $\text{rank}(\mY) \le r$. Suppose DGD+LOCAL~\eqref{eq:DGDtemplatePCA} is used to solve problem~\eqref{eq:rankrBMconsensus}, with weights $\{\widetilde w_{ji}\}$ and stepsize
\begin{equation}
\mu < \frac{1 - 2\omega}{\max_j \; (212 + 64\pi)\rho^2+34\|\mY_j\|_F  + \frac{(4+4\pi)}{\rho^2} \| \mY_j \|_F^2}
\label{eq:stepsize requirementballMFfj}
\end{equation}
for some $\rho > 0$ and where $\omega := \max_{j}\sum_{i\neq j} \widetilde w_{ji}< \frac{1}{2}$. Suppose the $J \times J$ connectivity matrix $\mW = \{w_{ji}\}$ (with $w_{ji}$ defined in~\eqref{eq:wtildetow}) is connected and symmetric. Let $\{ \vz(k) \}$ be the sequence generated by the DGD+LOCAL algorithm. Suppose $\vz(0)$ is chosen randomly from a probability distribution supported on a set $S \subseteq B_\rho$ with $S$ having positive measure, and suppose that under such random initialization, there is a positive probability that the sequence $\{\vz(k)\}$ remains bounded in $B_\rho$ and all limit points of $\{\vz(k)\}$ are in $B_\rho$.

Then conditioned on observing that $\{\vz(k)\} \subseteq B_\rho$ and all limit points of $\{\vz(k)\}$ are in $B_\rho$, DGD+LOCAL almost surely converges to a solution $\vz^\star = (\mU^{1\star},\ldots,\mU^{J\star},\mV_1^\star,\ldots,\mV_J^\star)$ with the following properties:
\begin{itemize}
\item Consensus: $\mU^{1\star} = \cdots = \mU^{J\star} = \mU^\star$.
\item Global optimality: $(\mU^\star,\mV^\star)$ is a global minimizer of~\eqref{eq:rankrBM}, where $\mV^\star$ denotes the concatenation of $\mV_1^\star,\ldots,\mV_J^\star$ as in~\eqref{eq:vpart}.
\end{itemize}
\label{thm:mainRevisedfj}
\end{thm}

\begin{proof}
We begin by arguing that DGD+LOCAL converges almost surely (when $\vz(0)$ is chosen randomly inside $B_\rho$) to a second-order critical point of~\eqref{eq:DGD problem PCA}. To do this, our goal is to invoke Theorem~\ref{thm:dgdconvergegenericballfj}. We note that each $f_j$ defined in~\eqref{eq:xyfj} satisfies $\inf_{\mU,\mV_j} f_j>-\infty$ and is twice-continuously differentiable. Also, since the functions $f_j$ are semi-algebraic, $g$ satisfies the {\L}ojasiewicz inequality globally. The functions $f_j$ do not have globally Lipschitz gradient. However, we can find quantities $L_{0,j}$, $L_{1,j}$, $L_{2,j}$ such that $\left| f_j\left(\vx,\vy_j\right)\right| \leq L_{0,j}$, $\left\Vert \nabla f_j\left(\vx,\vy_j\right)\right\Vert \leq L_{1,j}$, and $\left\Vert \nabla^{2}f_j(\vx,\vy) \right\Vert _{2}\leq L_{2,j}$ for all $(\vx,\vy_j) \in B_{2\rho}$. For $L_{0,j}$:
\begin{align*}
\left| f_j\left(\vx,\vy_j\right)\right| &= \| \mU\mV_j^\T - \mY_j \|_F^2 \\
&\le ( \| \mU\mV_j^\T\|_F + \| \mY_j \|_F)^2 \\
&\le ( \| \mU \|_F \| \mV^j \|_F + \| \mY_j \|_F)^2 \\
&\le ( 4\rho^2 +  \| \mY_j \|_F)^2 \\
&\le 32 \rho^4 + 2\| \mY_j \|_F^2.
\end{align*}
For $L_{1,j}$:
\begin{align*}
\left\Vert \nabla f_j \left(\vx,\vy_j\right)\right\Vert &= \left\| \begin{bmatrix} \nabla_{\mU} \| \mU\mV_j^\T - \mY_j \|_F^2 \\ \nabla_{\mV_j} \| \mU\mV_j^\T - \mY_j \|_F^2 \end{bmatrix} \right\|_F \\
&= \left\| \begin{bmatrix} 2 (\mU \mV_j^\T - \mY_j) \mV_j \\ 2 (\mU \mV_j^\T - \mY_j)^\T \mU \end{bmatrix} \right\|_F \\
&\le 2 \left( \|\mU \mV_j^\T \mV_j \|_F + \| \mY_j \mV_j \|_F + \| \mV_j \mU^\T \mU \|_F + \| \mY_j^\T \mU \|_F \right) \\
&\le 2 \left( 8 \rho^3 + 2 \rho \| \mY_j \|_F +  8 \rho^3 + 2 \rho \| \mY_j \|_F \right) \\
&= 32 \rho^3 + 8 \rho \| \mY_j \|_F.
\end{align*}
For $L_{2,j}$, we can bound the Lipschitz constant of $\nabla f_j$ in $B_{2\rho}$ as follows. Denote  $\mD=\begin{bmatrix}\mD_{\mU}\\ \mD_{\mV_j}\end{bmatrix}$. Then
\begin{align*}
&\frac{1}{2}\|\nabla^2 f_j(\mU,\mV_j)\|=\frac{1}{2}\max_{\|\mD\|_F=1}[\nabla^2 f_j(\mU,\mV_j)](\mD,\mD)\\
&=\max_{\|\mD\|_F=1} \|\mD_{\mU}\mV_j^\T + \mU\mD_{\mV_j}^\T\|_F^2+2\lg \mU\mV_j^\T,\mD_{\mU}\mD_{\mV_j}^\T\rg-2\lg \mY_j,\mD_{\mU}\mD_{\mV_j}^\T \rg\\
&\le \max_{\|\mD\|_F=1}
\frac{5}{2}(\|\mV_j\|_F^2+\|\mU\|_F^2)(\|\mD_{\mU}\|_F^2 +\|\mD_{\mV_j}\|_F^2)+\|\mY_j\|_F(\|\mD_{\mU}\|_F^2+\|\mD_{\mV_j}\|_F^2)\\
&\le \max_{\|\mD\|_F=1}
(10\rho^2+\|\mY_j\|_F)(\|\mD_{\mU}\|_F^2+\|\mD_{\mV_j}\|_F^2)=10\rho^2+\|\mY_j\|_F,
\end{align*}
where the last inequality holds because $\| \mU \|_F^2 + \|\mV_j\|_F^2 \le 4\rho^2$. Therefore we can bound  the Lipschitz constant of $\nabla f_j$ as $L_j\le 20\rho^2+2\|\mY_j\|_F$ for all $(\mU,\mV_j)$ such that $\| \mU \|_F^2 + \|\mV_j\|_F^2 \le 4\rho^2$. Now,
\begin{align*}
L_{2,j}+\frac{4L_{1,j}}{\rho}+\frac{\left(2+2\pi\right)L_{0,j}}{\rho^{2}}
&= 20\rho^2+2\|\mY_j\|_F + \frac{4}{\rho} (32 \rho^3 + 8 \rho \| \mY_j \|_F) + \frac{\left(2+2\pi\right)}{\rho^{2}} (32 \rho^4 + 2\| \mY_j \|_F^2) \\
&= 20\rho^2+2\|\mY_j\|_F + 128 \rho^2 + 32 \| \mY_j \|_F  + (64+64\pi) \rho^2 + \frac{(4+4\pi)}{\rho^2} \| \mY_j \|_F^2 \\
&= (212 + 64\pi)\rho^2+34\|\mY_j\|_F  + \frac{(4+4\pi)}{\rho^2} \| \mY_j \|_F^2.
\end{align*}
Thus, choosing $\mu$ to satisfy~\eqref{eq:stepsize requirementballMFfj} ensures that~\eqref{eq:stepsize requirementballfj} is met.

From Theorem~\ref{thm:dgdconvergegenericballfj}, we then conclude that conditioned on observing that $\{\vz(k)\} \subseteq B_\rho$ and all limit points of $\{\vz(k)\}$ are in $B_\rho$, DGD+LOCAL converges to a critical point of the objective function in~\eqref{eq:DGD problem PCA}, and the probability that this critical point is a strict saddle point is zero. We refer to this point as $\vz^\star$.

Next, note that the assumption of Proposition~\ref{prop:min f = min g for DGD} is satisfied if $\mY$ has rank at most $r$. In particular, there exist $\widetilde{\mU}, \widetilde{\mV}$ such that $\widetilde{\mU} \widetilde{\mV}^\T = \mY$ and so we may take $\vx^\star = \vec(\widetilde{\mU})$ and $\vy^\star_j = \vec(\widetilde{\mV}_j)$ to achieve $f_j(\vx^\star,\vy^\star_j) = 0$, which is the smallest possible value for each $f_j$. Proposition~\ref{prop:min f = min g for DGD} thus guarantees that~\eqref{eq:DGD problem PCA} has at least one critical point that is not a strict saddle (and in fact that it is a global minimizer that falls on the consensus subspace).

Next, note that the symmetric property required for Proposition~\ref{prop:nospuriouscritical for DGD} is satisfied. To see this, observe that
\[
\nabla_{\mU} \| \mU \mV_j^\T - \mY_j \|_F^2 = 2 (\mU \mV_j^\T - \mY_j) \mV_j
\]
and
\[
\nabla_{\mV_j} \| \mU \mV_j^\T - \mY_j \|_F^2 = 2 (\mU \mV_j^\T - \mY_j)^\T \mU.
\]
Thus,
\[
\langle \nabla_{\mU} \| \mU \mV_j^\T - \mY_j \|_F^2, \mU \rangle = 2\cdot\text{tr}(U^\T (\mU \mV_j^\T - \mY_j) \mV_j) = 2\cdot\text{tr}(\mV_j^\T (\mU \mV_j^\T - \mY_j)^\T \mU) = \langle \nabla_{\mV_j} \| \mU \mV_j^\T - \mY_j \|_F^2 , \mV_j \rangle.
\]
Proposition~\ref{prop:nospuriouscritical for DGD} thus guarantees that~\eqref{eq:DGD problem PCA} has no critical points outside of the consensus subspace. Since we have argued that DGD+LOCAL converges to a second-order critical point $\vz^\star$ of~\eqref{eq:DGD problem PCA}, it follows that $\vz^\star$ must be on the consensus subspace; that is, $\vz^\star = (\mU^{1\star},\ldots,\mU^{J\star},\mV_1^\star,\ldots,\mV_J^\star)$ with $\mU^{1\star} = \cdots = \mU^{J\star} = \mU^\star$. 

Next, Theorem~\ref{thm:strict saddle preserved for DGD} guarantees that $\vz^\star$ (in which $\mU^{1\star} = \cdots = \mU^{J\star} = \mU^\star$) corresponds to a critical point $(\mU^\star, \mV^\star)$ of the centralized problem~\eqref{eq:rankrBMconsensus}, which is exactly equivalent to problem~\eqref{eq:rankrBM}. Here, $\mV^\star$ is the concatenation of $\mV_1^\star,\ldots,\mV_J^\star$ as in~\eqref{eq:vpart}. Theorem~\ref{thm:centerss} tells us that problem~\eqref{eq:rankrBM} has two types of critical points: global minimizers and strict saddles. If $(\mU^\star,\mV^\star)$ were a strict saddle point of~\eqref{eq:rankrBM}, Theorem~\ref{thm:strict saddle preserved for DGD} tells us that $\vz^\star$ must then be a strict saddle of~\eqref{eq:DGD problem PCA}. However, $\vz^\star$ is almost surely a second-order critical point of~\eqref{eq:DGD problem PCA}, where the Hessian has no negative eigenvalues. It follows that $(\mU^\star,\mV^\star)$ must almost surely be a global minimizer of problem~\eqref{eq:rankrBM}.
\end{proof}

\section*{Acknowledgements}

This work was supported by the DARPA Lagrange Program under ONR/SPAWAR contract N660011824020. The views, opinions and/or findings expressed are those of the author(s) and should not be interpreted as representing the official views or policies of the Department of Defense or the U.S. Government.

The authors gratefully acknowledge Waheed Bajwa, Haroon Raja, Clement Royer, and Stephen Wright for many informative discussions on nonconvex and distributed optimization.

\appendix

\section{Proof of Proposition~\ref{prop:lip}}
\label{sec:prooflip}

\begin{proof}
Let $L = \max_j L_j$ and
\[
\vdelta_{\vz} = (\vdelta_{\vx^1},\ldots,\vdelta_{\vx^J},\vdelta_{\vy_1},\ldots,\vdelta_{\vy_J}).
\]
First, for any $\vz$ and $\vdelta_{\vz}$, and using the symmetry of $\mW=\{w_{ij}\}$, we have
\begin{align*}
\nabla g(\vz + \vdelta_{\vz}) - \nabla g(\vz)
&=
\begin{bmatrix}
\nabla_{\vx} f_1(\vx^1+\vdelta_{\vx^1},\vy_1+\vdelta_{\vy_1})-\nabla_{\vx} f_1(\vx^1,\vy_1)+ 4\sum_{i=1}^J w_{1i}(\vdelta_{\vx^1}-\vdelta_{\vx^i})
\\
\vdots
\\
\nabla_{\vx} f_J(\vx^J+\vdelta_{\vx^J},\vy_J+\vdelta_{\vy_J})-\nabla_{\vx} f_J(\vx^J,\vy_J)+ 4\sum_{i=1}^J w_{Ji}(\vdelta_{\vx^J}-\vdelta_{\vx^i})
\\
\nabla_{\vy} f_1(\vx^1+\vdelta_{\vx^1},\vy_1+\vdelta_{\vy_1})-\nabla_{\vy} f_1(\vx^1,\vy_1)
\\
\vdots
\\
\nabla_{\vy} f_J(\vx^J+\vdelta_{\vx^J},\vy_J+\vdelta_{\vy_J})-
\nabla_{\vy} f_J(\vx^J,\vy_J)
\end{bmatrix}
\end{align*}
Then with some rearrangement, denoting $\nabla f_j=\nabla_{\left[\begin{smallmatrix}
\vx \\ \vy
\end{smallmatrix}\right]}f_j$ and using the triangle inequality, we can obtain
\begin{align*}
\left\| \nabla g(\vz + \vdelta_{\vz}) - \nabla g(\vz) \right\|_2
&\leq  \left\|\begin{bmatrix}\nabla f_1(\vx^1+ \vdelta_{\vx^1},\vy_1 + \vdelta_{\vy_1}) - \nabla f_1(\vx^1,\vy_1) \\ \vdots \\ \nabla f_J(\vx^J+ \vdelta_{\vx^J},\vy_J + \vdelta_{\vy_J}) - \nabla f_J(\vx^J,\vy_J) \end{bmatrix}\right\|_2 +4 \left\|\begin{bmatrix} \sum_{i=1}^J w_{1i}(\vdelta_{\vx^1}-\vdelta_{\vx^i}) \\ \vdots\\ \sum_{i=1}^{J} w_{Ji}(\vdelta_{\vx^J}-\vdelta_{\vx^i})\end{bmatrix} \right\|_2\\
& \leq \sqrt{\sum_{j=1}^J  L_{j}^2\left\|\left[\begin{smallmatrix}
\vdelta_{\vx^j} \\ \vdelta_{\vy_j}
\end{smallmatrix}\right]\right\|_2^2} + 4\sqrt{\sum_{j=1}^J \left(\sum_{i=1}^J w_{ji}\right)^2 \left\|\vdelta_{\vx^j}\right\|_2^2}
+4\sqrt{\sum_{j=1}^J  \left\|\sum_{i= 1}^J w_{ji}\vdelta_{\vx^i}\right\|_2^2}\\
& \leq L \|\vdelta_{\vz}\|_F + \left(4\max_{j}\sum_{i=1}^J w_{ji}\right)\left\|\begin{bmatrix}
\vdelta_{\vx^1}&\cdots&\vdelta_{\vx^1}
\end{bmatrix} \right \|_F+4 \left(\max_{j}\sum_{i=1}^J   w_{ji} \right) \left\|\begin{bmatrix}
\vdelta_{\vx^1}&\cdots&\vdelta_{\vx^1}
\end{bmatrix} \right \|_F.
\end{align*}
where in the last line we use
\begin{align*}
\sqrt{\sum_{j=1}^J  \left\|\sum_{i= j}^J w_{ji}\vdelta_{\vx^i}\right\|_2^2}=\left\|\begin{bmatrix}
\vdelta_{\vx^1}&\cdots&\vdelta_{\vx^1}
\end{bmatrix}\mW\right\|_F
&=\left\|\mW^\T\begin{bmatrix}
\vdelta_{\vx^1}&\cdots&\vdelta_{\vx^1}
\end{bmatrix}^\T\right\|_F
\\
&\leq   \|\mW\| \left\|\begin{bmatrix}
\vdelta_{\vx^1}&\cdots&\vdelta_{\vx^1}
\end{bmatrix} \right \|_F
\\ &
\leq \left(\max_{j}\sum_{i=1}^J   w_{ji} \right) \left\|\begin{bmatrix}
\vdelta_{\vx^1}&\cdots&\vdelta_{\vx^1}
\end{bmatrix} \right \|_F
\end{align*}
since
$\|\mW\|\le \max_{j}\sum_{i\neq j}w_{ji}= \max_{j}\sum_{i=1}^Jw_{ji}$
in view of that $\mW$ is symmetric, $w_{ii}=0$ and $w_{ij}\ge 0$ by \eqref{eq:wtildetow}.

Finally, using the definition of $w_{ji}$  \eqref{eq:wtildetow}, we have $\max_{j}\sum_{i=1}^J   w_{ji}=\max_{j}\sum_{i\neq j}^J   w_{ji}=\max_j\frac{\sum_{i\neq j }^J  \widetilde w_{ji}}{4 \mu}=: \frac{\omega}{4\mu}$, and further by the inequality $\left\|\begin{bmatrix}
\vdelta_{\vx^1}&\cdots&\vdelta_{\vx^1}
\end{bmatrix} \right \|_F\le \|\vdelta_Z\|_F,$
we obtain that $\nabla g$ is Lipschitz continuous with constant
\[
L_{g} = L +  4  \left(\frac{\omega}{4\mu}\right)+4  \left(\frac{\omega}{4\mu}\right)= L + \frac{2\omega}{\mu}.
\]
\end{proof}

\section{Proof of Theorem~\ref{thm:boundedgd}}
\label{sec:proofboundedgd}

The proof involves constructing a function $\widetilde{h}$ such that $\widetilde{h}(\vx) = h(\vx)$ for all $\vx \in B_{\rho}$ but where $\widetilde{h}$ has a globally Lipschitz gradient.

To do this, first define a window function $w:\mathbb{R}^{n}\rightarrow\mathbb{R}$,
\[
w\left(\vx\right)=\begin{cases}
1, & \left\Vert \vx\right\Vert \leq\rho\\
2-\frac{\left\Vert \vx\right\Vert }{\rho}+\frac{1}{2\pi}\sin\left(\frac{2\pi\left\Vert \vx\right\Vert }{\rho}\right), & \rho<\left\Vert \vx\right\Vert <2\rho\\
0, & \left\Vert \vx\right\Vert \geq2\rho,
\end{cases}
\]
where $\left\Vert \cdot\right\Vert =\left\Vert \cdot\right\Vert_{2}$. Note also that \[
\nabla w\left(\vx\right)=\begin{cases}
0, & \left\Vert \vx\right\Vert \leq\rho\\
-\frac{2\vx}{\rho\left\Vert \vx\right\Vert }\sin^{2}\left(\frac{\pi\left\Vert \vx\right\Vert }{\rho}\right), & \rho<\left\Vert \vx\right\Vert <2\rho\\
0, & \left\Vert \vx\right\Vert \geq2\rho
\end{cases}
\]
and
\[
\nabla^{2}w\left(\vx\right)=\begin{cases}
0, & \left\Vert \vx\right\Vert \leq\rho\\
\left(\frac{2}{\rho\left\Vert \vx\right\Vert ^{3}}\sin\left(\frac{\pi\left\Vert \vx\right\Vert }{\rho}\right)-\frac{2\text{\ensuremath{\pi}}}{\rho^{2}\left\Vert \vx\right\Vert ^{2}}\sin\left(\frac{2\pi\left\Vert \vx\right\Vert }{\rho}\right)\right)\vx\vx^{\T}, & \rho<\left\Vert \vx\right\Vert <2\rho\\
0, & \left\Vert \vx\right\Vert \geq2\rho
\end{cases}
\]
since
\[
\frac{\partial^{2}w\left(\vx\right)}{\partial x_{i}\partial x_{j}}=\frac{2x_{i}x_{j}}{\rho\left\Vert \vx\right\Vert ^{3}}\sin\left(\frac{\pi\left\Vert \vx\right\Vert }{\rho}\right)-\frac{2\pi x_{i}x_{j}}{\rho^{2}\left\Vert \vx\right\Vert ^{2}}\sin\left(\frac{2\pi\left\Vert \vx\right\Vert }{\rho}\right).
\]
It is easy to verify that $w\in\mathcal{C}^{2}$ and $\left|w\left(\vx\right)\right|\leq1$. To bound the gradient $\nabla w$, we have
\[
\left\Vert \nabla w\right\Vert =\left\Vert -\frac{2\vx}{\rho\left\Vert \vx\right\Vert }\sin^{2}\left(\frac{\pi\left\Vert \vx\right\Vert }{\rho}\right)\right\Vert \leq\frac{2}{\rho}.
\]
For the Hessian $\nabla^{2}w$, we have
\[
\left\Vert \nabla^{2}w\right\Vert \leq\left\Vert \left(\frac{2}{\rho\left\Vert \vx\right\Vert ^{3}}\sin\left(\frac{\pi\left\Vert \vx\right\Vert }{\rho}\right)-\frac{2\pi}{\rho^{2}\left\Vert \vx\right\Vert ^{2}}\sin\left(\frac{2\pi\left\Vert \vx\right\Vert }{\rho}\right)\right)\vx^{\T}\vx\right\Vert \leq\frac{2+2\pi}{\rho^{2}}.
\]

Now, define
\[
\widetilde{h}\left(\vx\right)=h\left(\vx\right)w\left(\vx\right)=\begin{cases}
h\left(\vx\right), & \left\Vert \vx\right\Vert \leq\rho\\
h\left(\vx\right)\left(2-\frac{\left\Vert \vx\right\Vert }{\rho}+\frac{1}{2\pi}\sin\left(\frac{2\pi\left\Vert \vx\right\Vert }{\rho}\right)\right), & \rho<\left\Vert \vx\right\Vert <2\rho\\
0, & \left\Vert \vx\right\Vert \geq2\rho.
\end{cases}
\]
We have the following properties for $\widetilde{h}$:
\begin{itemize}
\item Since $h = \widetilde{h}$ in $B_\rho$, $\widetilde{h}$ satisfies the {\L}ojasiewicz inequality in $B_\rho$.
\item Since $h,w \in C^2$, $\widetilde{h} \in C^2$.
\item Since $\inf_{\R^n} h>-\infty$ and $\inf_{\R^n} w>-\infty$, $\inf_{\R^n} \widetilde{h}>-\infty$.
\item To globally bound the Lipschitz constant of the gradient of $\widetilde{h}$, note that
\begin{eqnarray*}
\left\Vert \nabla^{2}\widetilde{h}\right\Vert  & = & \left\Vert w\cdot\nabla^{2}h+\nabla h \cdot \left(\nabla w\right)^{\T}+\nabla w \cdot \left(\nabla h\right)^{\T}+h\cdot\nabla^{2}w\right\Vert \\
 & \leq & \left| w\right| \left\Vert \nabla^{2}h\right\Vert +2\left\Vert \nabla w\right\Vert \left\Vert \nabla h\right\Vert +\left| h\right| \left\Vert \nabla^{2}w\right\Vert \\
 & \leq & L_{2}+\frac{4L_{1}}{\rho}+\frac{\left(2+2\pi\right)L_{0}}{\rho^{2}}.
\end{eqnarray*}
\end{itemize}

Now consider the gradient descent algorithm with stepsize $\mu$ satisfying~\eqref{eq:steptilde}. Define
\begin{align*}
T_h = \{\vx(0) \in B_\rho:&~\text{all}~\{\vx(k)\} \subseteq B_\rho~\text{and all limit points of}~\{\vx(k)\}~\text{are in}~B_\rho\\
&~\text{when gradient descent is run on}~h~\text{starting at}~\vx(0)\}
\end{align*}
and
\begin{align*}
T_{\widetilde{h}} = \{\vx(0) \in B_\rho:&~\text{all}~\{\vx(k)\} \subseteq B_\rho~\text{and all limit points of}~\{\vx(k)\}~\text{are in}~B_\rho\\
&~\text{when gradient descent is run on}~\widetilde{h}~\text{starting at}~\vx(0)\}.
\end{align*}

Similarly, define
\[
\Sigma_h = \{\vx(0) \in B_\rho:~\{\vx(k)\}~\text{converges to a strict saddle when gradient descent is run on}~h~\text{starting at}~\vx(0)\}
\]
and
\[
\Sigma_{\widetilde{h}} = \{\vx(0) \in B_\rho:~\{\vx(k)\}~\text{converges to a strict saddle when gradient descent is run on}~\widetilde{h}~\text{starting at}~\vx(0)\}.
\]
Using the above properties, we see that Theorem~\ref{thm:avoid saddle gd} can be applied to $\widetilde{h}$, and so we conclude that $\Sigma_{\widetilde{h}}$ has measure zero.

Now, after running gradient descent on $h$ from a random initialization as in the theorem statement, condition on observing that $\{\vx(k)\} \subseteq B_\rho$ and all limit points of $\{\vx(k)\}$ are in $B_\rho$, i.e., that $\vx(0) \in T_h$. Because $\{\vx(k)\} \subseteq B_\rho$ and all limit points of $\{\vx(k)\}$ are in $B_\rho$, and because $\{\vx(k)\}$ matches the sequence that would be obtained by running gradient descent on $\widetilde{h}$, we can apply Theorem~\ref{thm:convergence gd with KL in ball} to conclude that $\{\vx(k)\}$ converges to a critical point of $\widetilde{h}$, and since this critical point belongs to $B_\rho$ and $\widetilde{h} = h$ inside $B_\rho$, we conclude that this is also a critical point of $h$.

Finally, using the definition of conditional probability, we have
\begin{align*}
P(\vx(0) \in \Sigma_h | \vx(0) \in T_h)
&= \frac{P( \vx(0) \in \Sigma_h \cap T_h )}{P(\vx(0) \in T_h)} \\
&= \frac{P( \vx(0) \in \Sigma_{\widetilde{h}} \cap T_{\widetilde{h}} )}{P(\vx(0) \in T_h)},
\end{align*}
where the second equality follows from the fact that $\widetilde{h} = h$ inside $B_\rho$: if a sequence of iterations stays bounded inside $B_\rho$ and converges to a strict saddle when gradient descent is run on $h$, the same will hold when gradient descent is run on $\widetilde{h}$, and vice versa. Since $\Sigma_{\widetilde{h}}$ has zero measure and because $\vx(0)$ is chosen randomly from a probability distribution supported on a set $S \subseteq B_\rho$ with $S$ having positive measure, $P( \vx(0) \in \Sigma_{\widetilde{h}} \cap T_{\widetilde{h}} ) = 0$. Also, by assumption, $P(\vx(0) \in T_h) > 0$. Therefore, $P(\vx(0) \in \Sigma_h | \vx(0) \in T_h) = \frac{0}{\text{nonzero}} = 0$.

\section{Proof of Proposition~\ref{prop:min f = min g for DGD}}
\label{sec:prooffgdgd}

\begin{proof}
	First note that
	\begin{align}
	&\min_{\vz}g(\vz) = \sum_{j=1}^J \left(f_j(\vx^j,\vy_j) + \sum_{i=1}^J w_{ji} \|\vx^j - \vx^i\|_2^2\right) \geq \sum_{j=1}^J \min_{\vx^j,\vy_j} f_j(\vx^j,\vy_j) = \sum_{j=1}^J f_j(\vx^\star,\vy_j^\star) = \min_{\vx,\vy} f(\vx,\vy).
	\label{eq:g>=fdgd}\end{align}
	On the other hand, we have
	\begin{align*}
	\min_{\vz}g(\vz) &= \min_{\vz}\sum_{j=1}^J \left(f_j(\vx^j,\vy_j) + \sum_{i=1}^J w_{ji} \|\vx^j - \vx^i\|_2^2\right)\\
	&\leq \min_{\vz:\vx^1=\cdots = \vx^J}\sum_{j=1}^J \left(f_j(\vx^j,\vy_j) + \sum_{i=1}^J w_{ji} \|\vx^j - \vx^i\|_2^2\right)  \\
	&= \min_{\vx,\vy}\sum_{j=1}^J f_j(\vx,\vy_j)  = \min_{\vx,\vy} f(\vx,\vy).
	\end{align*}
	Thus, we have
	\[
	\min_{\vz} g(\vz) = \min_{\vx,\vy} f(\vx,\vy).
	\]
	The proof is completed by noting that \eqref{eq:g>=fdgd} achieves the equality only at $\vz$ with $\vx^1 = \cdots = \vx^J$ since the topology defined by $\mW$ is connected.
\end{proof}

\section{Proof of Proposition~\ref{prop:nospuriouscritical for DGD}}
\label{sec:proofnscdgd}

\begin{proof}
	The critical points of the objective function in \eqref{eq:DGD problem} satisfy
	\begin{align}
	&\nabla_{\vx^j} g(\vz) = \nabla_{\vx}f_j(\vx^j,\vy_j) + \sum_{i=1}^J 2w_{ji}(\vx^j - \vx^i) = \mzero,	\label{eq:gradient g wrt xjdgd}	\\
	&\nabla_{\vy^j} g(\vz) = \nabla_{\vy_j}f_j(\vx^j,\vy_j) = \mzero, \forall \ j \in [J].
	\label{eq:gradient g wrt ydgd}	\end{align}
	
	Now taking the inner product of both sides in \eqref{eq:gradient g wrt xjdgd} with $\vx^j$ and also the inner product of both sides in \eqref{eq:gradient g wrt ydgd} with $\vy^j$ and using the property \eqref{eq:symmetric property for gradientdgd}, we have
	\[
\sum_{i=1}^J 2w_{ji}\langle \vx^j,  \vx^j - \vx^i\rangle = 0
	\]	
	for all $j \in [J]$. Using the symmetric property of $\mW$, we then have
	\[
\sum_{j=1}^J\sum_{i=1}^J w_{ji} \|\vx^j - \vx^i\|^2 = 0.
	\]
 Therefore,
\[
\vx^i = \vx^j, \ \text{if}\ w_{ij}\neq 0
\]
for any $i,j\in [J]$. Since the topology defined by $\mW$ is connected, we finally have
\[
\vx^1 = \cdots = \vx^J.
\]

\end{proof}

\section{Proof of Theorem~\ref{thm:strict saddle preserved for DGD}}
\label{sec:proofsspdgd}

\begin{proof} We rewrite $\setC_f$ as:
\begin{align*}
\setC_f = \left\{\vx,\vy: \sum_{j=1}^J \nabla_{\vx}f_j(\vx,\vy_j) = \vzero, \nabla_{\vy_j}f_j(\vx,\vy_j) = \vzero, \forall j\in[J] \right\}.
\end{align*}

The critical points of the objective function in \eqref{eq:DGD problem} satisfy
\begin{equation}\begin{split}
&\nabla_{\vx^j} g(\vz) = \nabla_{\vx}f_j(\vx^j,\vy_j) + \sum_{i=1}^J 2 (w_{ij} + w_{ji}) (\vx^j - \vx^i) = \mzero,\\
&\nabla_{\vy^j} g(\vz) = \nabla_{\vy_j}f_j(\vx^j,\vy_j) = \mzero, \forall \ j \in [J].
\end{split}
\nonumber\end{equation}
With this, we rewrite $\setC_g$ as
\begin{align*}
\setC_g = \bigg\{\vz: \nabla_{\vx}f_j(\vx^j,\vy_j) + \sum_{i=1}^J 2 (w_{ij} + w_{ji}) (\vx^j - \vx^i) = \mzero,  \nabla_{\vy_j}f_j(\vx^j,\vy_j) = \mzero, \forall \ j \in [J]\bigg\}.
\end{align*}
Thus, for any $\vz = (\vx^1,\ldots,\vx^J,\vy)\in\setC_g$ with $\vx^1 = \cdots = \vx^J = \vx$, we have that $ (\vx,\vy)$  is a critical point of \eqref{eq:centeralized problem}, i.e., $(\vx,\vy)\in\setC_f$. In what follows, we check how the Hessian information (especially the smallest eigenvalue of the Hessian) of $(\vx,\vy)$ is transformed to $\vz$.

At any point $(\vx,\vy)$, the Hessian quadratic form of $f$ for any $\vq_{\vx}$ and $\vq_{\vy} = \begin{bmatrix}\vq_{\vy_1}^\T & \cdots & \vq_{\vy_J}^\T\end{bmatrix}^\T$ is given by
\[
[\nabla^2 f(\vx,\vy)](\begin{bmatrix}\vq_{\vx}\\ \vq_{\vy} \end{bmatrix},\begin{bmatrix}\vq_{\vx}\\ \vq_{\vy} \end{bmatrix}) = \sum_{j=1}^J \nabla^2 f_j(\begin{bmatrix}\vq_{\vx}\\ \vq_{\vy_j} \end{bmatrix},\begin{bmatrix}\vq_{\vx}\\ \vq_{\vy_j} \end{bmatrix}).
\]
At any point $\vz$, the Hessian quadratic form of $g$ for any $\vq = \begin{bmatrix} \vq_{\vx^1}^\T & \cdots & \vq_{\vx^J}^\T & \vq_{\vy_1}^\T & \cdots & \vq_{\vy_J}^\T \end{bmatrix}$ is given by
\[
[\nabla^2 g(\vz)](\vq,\vq) = \sum_{j=1}^J \nabla^2 f_j(\begin{bmatrix}\vq_{\vx^j}\\ \vq_{\vy_j} \end{bmatrix},\begin{bmatrix}\vq_{\vx^j}\\ \vq_{\vy_j} \end{bmatrix}) + \sum_{j=1}^J 2 w_{ji} \|\vq_{\vx^j} - \vq_{\vx^i}\|^2_2.
\]

Now suppose $\lambda_{\min}(\nabla^2 f(\vx,\vy))<0$ (where $\lambda_{\min}$ denotes the smallest eigenvalue), i.e., there exist $\vq_{\vx},\vq_{\vy}$ such that $[\nabla^2 f(\vx,\vy)](\begin{bmatrix}\vq_{\vx}\\ \vq_{\vy} \end{bmatrix},\begin{bmatrix}\vq_{\vx}\\ \vq_{\vy} \end{bmatrix})<0$. Choosing $\vq_{\vx^1} = \cdots = \vq_{\vx^J}  = \vq_{\vx}$, we have $[\nabla^2 g(\vz)](\vq,\vq) <0$, i.e., $\lambda_{\min}(\nabla^2 g(\vz))<0$.
\end{proof}

\section{Proof of Theorem~\ref{thm:centerss}}
\label{sec:geometry for pca}
Denote by $h(\mU,\mV) = \frac{1}{2}\| \mU\mV^\T - \mY \|_F^2$. Let $\calC$ denote the set of critical points of $h$:
\[
\calC = \left\{(\mU,\mV): (\mU\mV^\T - \mY)\mV = \vzero, \ (\mU\mV^\T - \mY)^\T\mU = \vzero  \right\}.
\]
Our goal is to characterize the behavior of all the critical points that are not global minima. In particular, we want to show that every critical point of $h$ is either a global minimum or a strict saddle. Towards that end, we first recall the following result concerning the degenerate critical points.
\begin{lem}\cite[Theorem 8 with $\mX=\mId$]{nouiehed2018learning}
	Any pair $(\mU,\mV)\in\calC$ that is degenerate (i.e., $\rank(\mU\mV^\T)<r$) is either a global minimum of $h$ (i.e., $\mU\mV^\T = \mY_r$ where $\mY_r$ is a rank-$r$ approximation of $\mY$)  or a strict saddle (i.e., $\lambda_{\min}(\nabla^2 h(\mU,\mV))<0$).
	\label{lem:degenerate case}\end{lem}
Note that the above result holds for any matrix $\mY$. When $\rank(\mY)\leq r$, then $\mY_r = \mY$.  It follows from \Cref{lem:degenerate case} that the behavior of all degenerate critical points is quite clear. For the remaining non-degenerate critical points, using the same argument in~\cite[Theorems 2--4]{zhu2017global}, we first establish the following results concerning the critical points that are also balanced (i.e., $\mU^\T\mU = \mV^\T\mV$).

\begin{lem}\cite[Theorems 2--4]{zhu2017global}
Any pair $(\mU,\mV)\in\calC$ satisfying $\mU^\T\mU = \mV^\T\mV$ is either a global minimum of $h$ or a strict saddle.
\label{lem:balanced case}\end{lem}
The above result also holds for any matrix $\mY$. With this result, we now show that non-degenerate critical points behave similarly to degenerate ones.

\begin{lem}
Any pair $(\mU,\mV)\in\calC$ that is non-degenerate (i.e., $\rank(\mU\mV^\T)=r$) is either a global minimum of $h$ or a strict saddle.
\label{lem:non-degenerate case}\end{lem}
\begin{proof}[Proof of \Cref{lem:non-degenerate case}] Suppose $(\mU,\mV)$ is not a global minimum of $h$.
Let $\mU\mV^\T = \mP\mSigma \mQ^\T$ be a reduced SVD of $\mU\mV^\T$. Since $\rank(\mU\mV^\T) =r$ and both $\mU$ and $\mV$ have only $r$ columns, we know $\rank(\mU) = \rank(\mV) = r$. Denote by $\mD = (\mU^\T\mU)^{-1}\mU^\T\mP\mSigma^{1/2}$ and $\mG = (\mV^\T\mV)^{-1}\mV^\T\mQ\mSigma^{1/2}$. With this, we have
\[
\mD\mG^\T = (\mU^\T\mU)^{-1}\mU^\T\mP\mSigma\mQ^\T\mV (\mV^\T\mV)^{-1} = \mId,
\]
and
\[
\widetilde\mU = \mU \mD = \mP\mSigma^{1/2}, \ \widetilde\mV = \mV \mG = \mQ\mSigma^{1/2}.
\]
The above constructed pair $(\widetilde\mU,\widetilde\mV)$ satisfies
\[
\widetilde \mU \widetilde \mV^\T = \mU\mV^\T, \ \widetilde \mU^\T \widetilde \mU = \widetilde \mV^\T \widetilde \mV.
\]
Since $(\mU,\mV)\in\calC$, we have
\[
\nabla h_{\mU}(\widetilde \mU, \widetilde \mV)= \nabla h_{\mU}(\mU,\mV)\mD = \vzero, \ \nabla h_{\mV}(\widetilde \mU, \widetilde \mV)= \nabla h_{\mV}(\mU,\mV)\mG = \vzero,
\]
which implies that $(\widetilde\mU,\widetilde\mV)$ is also a critical point (but not a global minimum since by assumption $(\mU,\mV)$ is not a global minimum) of $h$. Since $(\widetilde\mU,\widetilde\mV)$ is also balanced, it follows from \Cref{lem:balanced case} that there exists $\widetilde\mDelta_{\widetilde \mU}$ and $\widetilde\mDelta_{\widetilde \mV}$ such that
\[
[\nabla^2h(\widetilde\mU,\widetilde\mV)](\widetilde\mDelta,\widetilde\mDelta) <0.
\]
Now construct $\mDelta_{\mU} = \mDelta_{\widetilde \mU} \mD^{-1}$ and $\mDelta_{\mV} = \widetilde\mDelta_{\widetilde \mV} \mG^{-1}$. Then, we have
\begin{align*}
[\nabla^2h(\mU,\mV)](\mDelta,\mDelta) &= \|\mDelta_{\mU}\mV^\T + \mU \mDelta_{\mV}^\T\|_F^2 + 2\langle \mU\mV^\T - \mY, \mDelta_{\mU}\mDelta_{\mV}^\T \rangle \\
&= \|\widetilde\mDelta_{\widetilde\mU}\widetilde\mV^\T + \widetilde\mU \mDelta_{\widetilde\mV}^\T\|_F^2 + 2\langle \widetilde\mU\widetilde\mV^\T - \mY, \mDelta_{\widetilde\mU}\mDelta_{\widetilde\mV}^\T \rangle\\
& = [\nabla^2h(\widetilde\mU,\widetilde\mV)](\widetilde\mDelta,\widetilde\mDelta)<0,
\end{align*}
which implies that $(\mU,\mV)$ is a strict saddle.
\end{proof}

\Cref{lem:balanced case} together with \Cref{lem:non-degenerate case} implies that any pair $(\mU,\mV)\in\calC$ is either a global minimum of $h$ or a strict saddle.

\bibliographystyle{abbrv}
\bibliography{nonconvex}

\end{document}